\documentclass[a4paper,10pt]{amsart}

\usepackage{amsbsy}
\usepackage{amsmath}
\usepackage{amsfonts}
\usepackage{braket}
\usepackage{latexsym}
\usepackage{amsthm}
\usepackage{amsxtra}
\usepackage{amssymb}
\usepackage{mathrsfs}
\usepackage{verbatim}
\usepackage{epsfig}
\usepackage{bm}
\usepackage{xypic}

\usepackage[usenames,dvipsnames]{xcolor}

\usepackage[shortlabels]{enumitem}
\SetEnumerateShortLabel{a}{\textup{(\alph*)}}
\SetEnumerateShortLabel{A}{\textup{(\Alph*)}}
\SetEnumerateShortLabel{1}{\textup{(\arabic*)}}
\SetEnumerateShortLabel{i}{\textup{(\roman*)}}
\SetEnumerateShortLabel{I}{\textup{(\Roman*)}}

\usepackage{hyperref} 

\usepackage[initials,lite,alphabetic]{amsrefs} 

\theoremstyle{plain}
\newtheorem{thm}{Theorem}[section]

\newtheorem{lem}[thm]{Lemma}

\newtheorem{prop}[thm]{Proposition}

\newtheorem*{proper*}{Property}

\newtheorem{cor}[thm]{Corollary}

\newtheorem{conj}[thm]{Conjecture}
\newtheorem{ques}[thm]{Question}

\newtheorem*{lm*}{Lemma}
\newtheorem*{thm*}{Theorem}

\theoremstyle{definition}

\newtheorem{df}[thm]{Definition}
\newtheorem*{df*}{Definition}

\newtheorem{exam}[thm]{Example}

\newtheorem{ex-notn}[thm]{Example/Notation}

\newtheorem{con}[thm]{Construction}

\theoremstyle{remark}

\newtheorem{rem}[thm]{Remark}

\newtheorem*{acknowledgement*}{Acknowledgement}

\newtheorem*{ex*}{Example}
\newtheorem*{exer*}{Exercise}
\newtheorem*{rem*}{Remark}
\newtheorem*{prob*}{Problem}
\newtheorem*{prop*}{Proposition}

\def\cone{\operatorname{cone}}

\def\gr{\operatorname{gr}}

\def\initial{\operatorname{in}}

\def\lcm{\operatorname{lcm}}

\def\KK{\mathbb{K}}

\def\calC{\mathcal{C}}

\def\calF{\mathcal{F}}
\def\calG{\mathcal{G}}

\def\calI{\mathcal{I}}

\def\calL{\mathcal{L}}

\def\calP{\mathcal{P}}

\def\calR{\mathcal{R}}

\def\calV{\mathcal{V}}

\def\bdf{{\bm f}}

\def\bdT{{\bm T}}

\def\bdU{{\bm U}}

\def\bdw{{\bm w}}

\def\bdx{{\bm x}}

\def\bdalpha{ {\bm \alpha}}
\def\bdbeta{ {\bm \beta}}

\def\alert#1{\textcolor{red}{#1}}

\def\deg{\operatorname{deg}}

\def\divides{\operatorname{|}}

\def\Index#1{\emph{#1}}

\def\notdivides{{\not |\,\,}}
\def\onto{\twoheadrightarrow}

\def\sqr#1#2{{\vcenter{\hrule height.#2pt
\hbox{\vrule width.#2pt height#1pt \kern#1pt
\vrule width.#2pt}
\hrule height.#2pt}}}

\DeclareMathOperator\Sym{ Sym}

\def\opn#1#2{\def#1{\operatorname{#2}}} 
\opn\lex{lex}
\opn\rev{rev}
\opn\Lex{Lex}
\opn\GL{GL}

\def\Grobner{Gr\"{o}bner}

\begin{document}
\title{On a class of squarefree monomial ideals of linear type}
\author{Yi-Huang Shen}
\address{Wu Wen-Tsun Key Laboratory of Mathematics of CAS and School of Mathematical Sciences, University of Science and Technology of China, Hefei, Anhui, 230026, People's Republic of China}
\thanks{This work is supported by the National Natural Science Foundation of China (11201445).  Helpful comments from Louiza Fouli and Kuei-Nuan Lin during the preparation of this work  are gratefully appreciated.}
\email{yhshen@ustc.edu.cn}
\begin{abstract}
  In a recent work, Fouli and Lin generalized a Villarreal's result and showed that
  if each connected components of the line graph of a squarefree monomial ideal contains at most a unique odd cycle, then this ideal is of linear type.
  In this short note, we reprove this result with Villarreal's original ideas together with a method of Conca and De Negri. We also propose a class of squarefree monomial ideals of linear type.
\end{abstract}
\maketitle

\section{Introduction}
Let $S$ be a Noetherian ring and $I$ an $S$-ideal. The \Index{Rees algebra of $I$} is the subring of the ring of polynomials $S[t]$
\[
\calR(I):=S[It]=\oplus_{i\ge 0}I^i t^i. 
\]

Analogously, one has $\Sym(I)$, the \Index{symmetric algebra of $I$} which is obtained from the tensor algebra of $I$ by imposing the commutative law. 
The symmetric algebra $\Sym(I)$ is equipped with an $S$-Module homomorphism $\pi: I\to \Sym(I)$ which solves the following universal problem. For a commutative $S$-algebra $B$ and any $S$-module homomorphism $\varphi: I\to B$, there exists a unique $S$-algebra homomorphism $\Phi:\Sym(I)\to B$ such that the diagram
\[
\xymatrix{
I \ar[r]^\varphi \ar[d]_\pi & B \\
\Sym(I) \ar[ur]_\Phi
}
\]
is commutative. 

Both Rees algebras and symmetric algebras have been studied by many authors from different point of views. For instance, it is known that there is a canonical surjection $\alpha:\Sym(I)\onto \calR(I)$. When $S$ is an integral domain, the kernel of $\alpha$ is just the $S$-torsion submodule of $\Sym(I)$ (cf.~\cite[page 3]{MR1275840}).
The main purpose of this article is to investigate when the canonical map $\alpha$ is an isomorphism; whence, $I$ is called an ideal of \Index{linear type}. 

Suppose $I=\braket{f_1,\dots,f_s}$ and consider the presentation $\psi: S[\bdT]:=S[T_1,\dots,T_s] \to S[It]$ defined by setting $\psi(T_i)=f_it$. Since this map is homogeneous, the kernel $J=\bigoplus_{i\ge 1}J_i$ is a graded ideal; it will be called the defining ideal of $\calR(I)$ (with respect to this presentation).  Since the linear part $J_1$ generates the defining ideal of $\Sym(R)$ (cf.~\cite[page 2]{MR1275840}), $I$ is of linear type if and only if $J=\braket{J_1}$.

From now on, we assume that $S=\KK[x_1,\dots,x_n]$ is a polynomial ring over the field $\KK$ and $I$ is a monomial ideal. In this case, Conca and De Negri introduced the notion of $M$-sequence (which we will investigate in section 2) and showed in \cite[Theorem 2.4.i]{MR1666661} that if $I$ is generated by an $M$-sequence of monomials, then $I$ is of linear type,

If in addition $I$ is squarefree, it can be realized as the facet ideal of some simplicial complex $\Delta$ with vertex set $[n]:=\Set{1,2,\dots,n}$. Let $\calF(\Delta)$ be the set of facets of $\Delta$.  Recall that a facet $F\in \calF(\Delta)$ is called a \Index{(simplicial) leaf} if either $F$ is the only facet of $\Delta$ or there exists a distinct $G\in \calF(\Delta)$ such that $H\cap F \subset G\cap F$ for each $H\in \calF(\Delta)$ with $H\ne F$. The simplicial complex $\Delta$ is called a \Index{(simplicial) forest} if each subcomplex of $\Delta$ still has a leaf. And a connected forest is called a \Index{(simplicial) tree}. In this situation, Soleyman Jahan and Zheng \cite[Theorem 1.14]{MR2968912} showed that $I$ is generated by an $M$-sequence if and only if $\Delta$ is a forest. Thus, the facet ideal of forests are of linear type.

Squarefree monomial ideals of degree $2$ can also be realized as the edge ideal of some finite simple graph $G$. When this $G$ is connected, Villarreal \cite[Corollary 3.2]{MR1335312} showed that $I$ is of linear type if and only if $G$ is a tree or contains a unique odd cycle. Fouli and Lin generalized this pattern to higher degrees. Let $G=L(I)$ be the \Index{line graph} of $I$. Then the vertices $v_i$ of $G$ correspond to the minimal monomial generators $f_i$ of $I$ respectively, and $\Set{v_i,v_j}$ is an edge of $G$ if and only if $\gcd(f_i,f_j)\ne 1$. This $G$ is also known as the \Index{generator graph} of $I$. If $G$ is a forest or each connected component of $G$ contains at most a unique odd cycle, Fouli and Lin \cite[Theorem 3.4]{arXiv:1205.3127} showed that $I$ is of linear type. We will generalize this result and propose a new class of monomial ideals of linear type.

This paper is organized as follows. In section 2, we introduce the notion of $M$-element and show that the class of squarefree monomial ideals of linear type is closed under the operation of adding $M$-elements. In section 3, we consider the simplicial cycles, which are introduced by Caboara, Faridi and Selinger \cite{MR2284286}, and study the linear type property. In section 4,
we generalize some of Villarreal's results and ideas. In particular, we will reprove Fouli and Lin's results.
In Section 5, We introduce the Villarreal class of simplicial complexes and show that corresponding squarefree ideals are of linear type.  In the final section, we consider the notion of cycles introduced by other authors. 

\section{$M$-elements}
Throughout this section, let $I$ be a squarefree monomial ideal in $S=\KK[x_1,\dots,x_n]$ with the minimal monomial generating set $G(I)=\Set{f_1,\dots,f_s}$. We also write $I'=\braket{f_2,\dots,f_s}$.

\begin{df}
  \label{M-element}
  Following the spirit of Conca and De Negri \cite{MR1666661}, we say $f_1$ is an \Index{$M$-element} of $I$, if there exists a total order on the set of indeterminates that appear in $f_1$, say $x_1<\cdots <x_r$ with $f_1=x_1\cdots x_r$, such that whenever $x_k\mid f_j$ with $1\le k \le r$ and $1<j$, then $x_k\cdots x_r \mid f_j$. If $f_1$ is an $M$-element of $I$, we say $I$ is obtained from $I'$ by \Index{adding an $M$-element}. 
\end{df}

Thus, the sequence of squarefree monomials $f_1,\dots,f_s$ is an $M$-sequence in the sense of \cite{MR1666661} if and only if $f_i$ is an $M$-element of the ideal $\braket{f_i,f_{i+1},\dots,f_s}$ for each $i$ with $1\le i \le s$.

Recall that a variable $x_i$ is called a \Index{free variable} of $I$ if there exists an index $u\in [s]$ such that $x_i\divides f_u$ and $x_i\notdivides f_j$ for any $j\ne u$.  If $k=1$ in the Definition \ref{M-element}, then $f_1=x_1\cdots x_r\mid f_j$, which contradicts the assumption that $G(I)=\Set{f_1,\dots,f_r}$ is the minimal monomial generating set of $I$. Thus, a necessary condition for $f_1$ to be an $M$-element is that $f_1$ contains a free variable.

Now, consider the presentation 
\[
\psi: S[\bdT] \to \calR(I)
\]
defined by setting $\psi(T_i)=f_i t$ and denote by $J$ the kernel of $\psi$.
Similarly, we consider the presentation
\[
\psi':S[\bdT']:=S[T_2,\dots,T_s] \to \calR(I')
\]
defined by setting $\psi'(T_i)=f_it$ and denote by $J'$ the kernel of $\psi'$.
Obviously $J'S[\bdT]\subseteq J$. Following \cite{MR1666661}, we set
\[
m_{ij}=f_i/\gcd(f_i,f_j) \quad \text{ and } \quad l_{ij}=m_{ij}T_j - m_{ji}T_i
\]
for all $1\le i<j \le s$.
Let $\tau$ be a monomial order on $S[\bdT]$ such that the initial term $\initial_\tau(l_{ij})=m_{ij}T_j$; for instance, one can take the lexicographic order induced by the total order $T_s>\cdots > T_1$. This $\tau$ induces a monomial order on $S[\bdT']$, which we shall denote by $\tau'$.

\begin{prop}
  [{Essentially \cite[Theorem 2.4.i]{MR1666661}}]
  \label{CDN}
 Suppose $f_1$ is an $M$-element and $\Set{v_1,\dots,v_r}$ form a {\Grobner} basis of $J'$ with respect to $\tau'$. Then
 \[
 \Set{l_{1j}\mid 2\le j \le s} \cup \Set{v_1,\dots,v_r}
 \]
 form a {\Grobner} basis of $J$ with respect to $\tau$.
\end{prop}

\begin{proof}
  We argue by contradiction. Suppose the claim is false. Write 
  \[
  Q=\Set{m_{1j}T_j \mid 2\le j \le s} \cup \Set{\initial_{\tau'}(v_j)\mid 1\le j \le r}.
  \]
  Notice that $J$ has a universal {\Grobner} basis consisting of binomial relations by \cite[Lemma 2.2]{MR1104431}. Thus, it suffices to consider a binomial relation $f:=a\bdT^\bdalpha -b\bdT^\bdbeta\in J$ with $a,b\in S$ being monomials and the initial monomials of $f$ being not in $Q$. As a matter of fact, we may further assume that 
  \begin{enumerate}[1]
    \item $\gcd(a\bdT^\alpha,b\bdT^\beta)=1$;
    \item \label{assum2}
      neither $a\bdT^\alpha$ nor $b\bdT^\beta$ is divisible by any monomial in $Q$.
  \end{enumerate}

  Let $i$ be the smallest index such that $T_i$ appears in $\bdT^\alpha$ or in $\bdT^\beta$. If $i\ne 1$, then $f\in J'$. This contradicts to the choice of $\Set{v_1,\dots,v_t}$. If $i=1$, by symmetry, we may assume that $T_1\divides \bdT^\bdalpha$. Consequently, $f_1\divides a\psi(\bdT^\bdalpha)=b\psi(\bdT^\bdbeta)$. We have two cases.
  \begin{enumerate}[i]
    \item If $f_1\divides b$, we can let $T_j$ be any of the indeterminates in $\bdT^\bdbeta$. Then $m_{1j}T_j \divides f_1T_j \divides b\bdT^\bdbeta$. Since $1<j$, this contradicts the assumption \ref{assum2} above.
    \item Otherwise, $f_1=x_1\cdots x_r\notdivides b$. Suppose the indeterminates are ordered as in Definition \ref{M-element} and $k\in [r]$ is the minimal index such that $x_k\notdivides b$. If $k\ge 2$, then $x_1\cdots x_{k-1}\divides b$. On the other hand, since $x_k\divides f_1 \divides b\psi(\bdT^\bdbeta)$, there is some $j$ such that $T_j\divides \bdT^{\bdbeta}$ and $x_k\divides f_j$. Since $1<j$ and $f_1$ is an $M$-element, one has $x_k\cdots x_r \divides f_j$. Consequently, $x_k\cdots x_r \divides \gcd(f_1,f_j)$ and $m_{1j}\divides x_1\cdots x_{k-1}\divides b$. It follows that $m_{1j}T_j\divides b\bdT^\bdbeta$. This contradicts the assumption \ref{assum2} again. 
      
      If instead $k=1$, we will find a similar $j>1$ and conclude that $f_1=x_1\cdots x_r \divides f_j$. This contradicts the minimality of $G(I)$.
      \qedhere
  \end{enumerate}
\end{proof}

According to \cite{MR1666661}, the ideal $I$ is of \Index{{\Grobner} linear type} if the linear relations in the defining ideal $J$ form a {\Grobner} basis of $J$ (with respect to some monomial order on $S[\bdT]$).

\begin{cor}
  If $I'$ is of {\Grobner} linear type with respect to $\tau'$, then  $I$ is of {\Grobner} linear type with respect to $\tau$.
\end{cor}

\begin{cor}
  \label{res0}
  If $I'$ is of linear type, then so is $I$.
\end{cor}

\begin{proof}
  Following Proposition \ref{CDN}, each binomial relation $a\bdT^\bdalpha-b\bdT^\bdbeta\in J$ can be written as an $S[\bdT]$-linear combination of $\Set{l_{1j}\mid 2\le j \le s} \cup \Set{v_1,\dots,v_r}$.
  If $I'$ is of linear type, each $v_k$ in $\Set{v_1,\dots,v_r}$ can be written as an $S[\bdT']$-linear combination of $\Set{l_{ij}\in J' | 2\le i<j\le s}$. Thus $J$ is generated by its linear part and $I$ is of linear type.
\end{proof}

\begin{ques}
  If the ideal $I$ is of linear type, is $I'$ also of linear type?
\end{ques}

\begin{ques}
  If $\Gamma$ is a simplicial tree and $K$ is its facet ideal, then
  \begin{enumerate}[1]
    \item $K$ has sliding depth by \cite[Theorem 1]{Faridi2002};
    \item $K$ satisfies $\calF_1$ by \cite[Proposition 4]{Faridi2002};
    \item $K$ is sequentially Cohen-Macaulay by \cite[Corollary 5.6]{MR2043324};
    \item the Rees algebra of $K$ is normal and Cohen-Macaulay by \cite[Corollary 4]{Faridi2002}.
  \end{enumerate}
  Thus, it is a natural question to ask: if $I'$ is a squarefree monomial ideal that satisfies one of the above properties, will $I$ (obtained from $I'$ by adding $M$-elements) satisfy the same property as well?
\end{ques}

Let $\calL$ be the class of squarefree monomial ideals of linear type.

\begin{cor}
  \label{res1}
  The class $\calL$ is closed under the operation of adding $M$-elements.
\end{cor}

\begin{proof}
  This is a paraphrase of Corollary \ref{res0}.
\end{proof}

\section{Simplicial cycles}
The following notation and terminology regarding graphs and simplicial complexes will be fixed throughout this work.

Let $\calG$ be a simple graph. Following \cite{MR1335312}, a \Index{walk} of length $k$ in $\calG$ is an alternating sequence of vertices and edges $w=\Set{v_1,z_1,v_2,\dots,v_{k-1},z_k,v_k}$, where $z_i=\Set{v_{i-1},v_i}$ is the edge joining $v_{i-1}$ and $v_i$. The walk $w$ is \Index{closed} if $v_0=v_k$. A \Index{cycle} of length $k$ is a closed walk, in which the vertices are distinct. A closed walk of even length is called a \Index{monomial walk}.

Let $\Delta$ be a simplicial complex with vertex set $[n]$ and the set of facets $\calF(\Delta)=\Set{F_1,\dots,F_s}$.  For each facet $F\in \calF(\Delta)$, let $\bdx_F=\prod_{i\in F}x_i$. The ideal $I(\Delta)=\braket{\bdx_{F_1},\dots,\bdx_{F_s}}$ is the \Index{facet ideal} of $\Delta$. The facet set can also be treated as a \Index{clutter} $\calC$, namely, a hypergraph such that no edge of $\calC$ is properly contained in any other edge of $\calC$. In this situation, the facet ideal is also known as the \Index{edge ideal} or \Index{circuit ideal} of $\calC$. Throughout this paper, when the facet ideal (resp.~edge ideal) is of linear type, we shall say that the original simplicial complex (resp.~clutter) is of \Index{linear type}.

Following \cite{MR2284286}, two facets $F$ and $G$ of $\Delta$ are \Index{strong neighbors}, written $F\sim_\Delta G$, if $F\ne G$ and for all facets $H\in \Delta$, $F\cap G \subset H$ implies that $H=F$ or $H=G$.  The simplicial complex $\Delta$ is called a \Index{simplicial cycle} or simply a \Index{cycle} if $\Delta$ has no leaf but every nonempty proper subcomplex of $\Delta$ has a leaf.  This definition is more restrictive than the classic definition of  (hyper)cycles of hypergraphs due to Berge \cite[page 155]{MR1013569}. Simplicial cycles are minimal hypercycles in the sense that once a facet is removed, what remains is not a cycle anymore, and does note contain one. The following theorem characterizes the structure of simplicial cycles.

\begin{lem}
  [{\cite[Theorem 3.16]{MR2284286}}]
  \label{simpl_cycle}
  Let $\Delta$ be a simplicial complex. Then $\Delta$ is a simplicial cycle if and only if the facets of $\Delta$ can be written as a sequence of strong neighbors $F_1\sim_\Delta F_2 \sim_\Delta \cdots \sim_\Delta F_s \sim_\Delta F_1$ such that $s\ge 3$ and for all $i,j$,
  \[
  F_i\cap F_j =\bigcap_{k=1}^s F_k \qquad \text{if $j\not\equiv i-1,\, i,\, i+1\mod n$}.
  \]
\end{lem}

Let $\Delta$ be a simplicial complex. The line graph $L(\Delta):=L(I(\Delta))$ is a finite graph, whose vertices $v_i$ correspond to the facets $F_i$ of $\Delta$ respectively, and $\Set{v_i,v_j}$ is an edge of $G$ if and only if $F_i\cap F_j\ne \emptyset$. If $L(\Delta)$ is a cycle graph, we will call $\Delta$ a \Index{linear cycle}. 

\begin{rem}
  \label{cone} The previous lemma implies that simplicial cycles are either linear cycles or cones over such a structure.  If the intersection $\cap F_k$ is indeed not empty in the previous lemma, we may assume that $\Delta=\cone_v(\Delta')$, where $\Delta'$ is a simplicial complex with vertex set $[n]\setminus \Set{v}$. The $\Delta'$ is a simplicial cycle and the corresponding facet ideals satisfy $I(\Delta)=x_v I(\Delta')S$. Thus $I(\Delta)$ is of linear type if and only if $I(\Delta')$ is so.
\end{rem}

\begin{con}
  \label{push-down}
  Let $\Delta$ be a linear cycle and $I=I(\Delta)\subset S=\KK[x_1,\dots,x_n]$ its facet ideal. When the length is $3$, we additionally assume that the GCD of the monomial generators of $I$ is trivial, i.e., $\Delta$ is not a cone.  Now each indeterminate of the polynomial ring shows in at most two monomial generators of $I$. Consider a subset $D$ (for deletion) of these indeterminates defined as follows.
  \begin{enumerate}[i]
    \item If $x_i$ is a free variable, then $x_i\in D$;
    \item For each pair of monomial generators with a non-trivial common factor $f$, keep one indeterminate (say, $x_i$) of $f$ and take all remaining indeterminates dividing $f$ to be in $D$. In this case, the remaining indeterminates shall be called the \Index{shadows} of $x_i$.
  \end{enumerate}
  Now, write $D^c=\Set{x_1,\dots,x_n}\setminus D$ for the complement set and consider the ring homomorphism $\chi:S\to \KK[D^c]\subset S$ such that $\chi(x_i)=1$ for $x_i\in D$ and $\chi(x_i)=x_i$ for $x_i\in D^c$.  
\end{con}

\begin{rem}
  \label{connected-component}
  Suppose $R_1$ and $R_2$ are affine algebras over a field $\KK$ and let $I_1\subset R_1$ and $I_2\subset R_2$ be ideals. Let
  \[
  I=(I_1,I_2)\subset R=R_1\otimes_\KK R_2.
  \]
  Suppose that
  \[
  0\to L_1 \to R_1[T_1,\dots,T_m] \to R_1[f_1t,\dots,f_mt]\to 0,
  \]
  and
  \[
  0\to L_2 \to R_2[U_1,\dots,U_n] \to R_2[g_1t,\dots,g_nt]\to 0
  \]
  are algebra presentations of the Rees algebras. Then in the following presentation of $\calR(I)$
  \[
  0\to (L_1,L_2,J) \to R[\bdT,\bdU]\to \calR(I)\to 0,
  \]
  the additional generators $J$ can be generated by the obvious Koszul elements: $g_iT_j-f_jU_i$; see \cite[page 133]{MR1275840}.  Thus, to show that a simplicial complex is of linear type, it suffices to show that each of its connected components has this property.
\end{rem}

The following work of Fouli and Lin is a partial generalization of Villarreal's result \cite[Corollary 3.2]{MR1335312}:

\begin{prop}
  [{\cite[Proposition 3.3]{arXiv:1205.3127}}]
  \label{Fouli-Lin}
  Let $S$ be a polynomial ring over a field and let $I$ be a squarefree monomial ideal in $S$. If the line graph $L(I)$ of $I$ is a disjoint union of graphs with a unique odd cycle, then $I$ is an ideal of linear type.
\end{prop}

\begin{cor}
  \label{odd_cycles}
  Let $\Delta$ be a simplicial cycle of odd length $s\ge 3$.  Then the facet ideal $I=I(\Delta)$ is of linear type.
\end{cor}

\begin{proof}
  This follows from Proposition \ref{Fouli-Lin} and Remark \ref{cone}.
\end{proof}

Let $\calI_k$ be the set of non-decreasing sequence of integers in $\Set{1,2,\dots,s}$ of length $k$. If $\bdalpha=(i_1,i_2,\dots,i_k)\in \calI_k$, set $\bdf_\bdalpha=f_{i_1}\cdots f_{i_k}$ and $\bdT_\bdalpha=T_{i_1}\cdots T_{i_k}$. For every $\bdalpha,\bdbeta\in \calI_k$, set
  \[
  \bdT_{\bdalpha,\bdbeta}=\frac{\bdf_\bdbeta}{\gcd(\bdf_\bdalpha,\bdf_\bdbeta)}
  \bdT_\bdalpha-\frac{\bdf_\bdalpha}{\gcd(\bdf_\bdalpha,\bdf_\bdbeta)}\bdT_\bdbeta.
  \]
  It is well-known that the defining ideal $J$ is generated by these $\bdT_{\bdalpha,\bdbeta}$'s with $\bdalpha,\bdbeta\in \calI_k$ and $k\ge 1$ (cf.~ \cite{MR2611561}). Notice that when $k=1$, $\bdalpha=i$ and $\bdbeta=j$, then $\bdT_{\bdalpha,\bdbeta}=-l_{j,i}$ which is defined before Proposition \ref{CDN}.

\begin{prop}
  [Unwrapping local cones preserves linear-type property]
  \label{unwrap_cone}
  Let $\widetilde{I}$ be a squarefree monomial ideal in $\widetilde{S}=S[x_{n+1}]$ with minimal monomial generators $G(\widetilde{I})=\Set{\widetilde{f}_1,\dots,\widetilde{f}_s}$.  Substitute every $x_{n+1}$ with $1$:
  \[
  f_{i}=\left.\widetilde{f}_i\right|_{x_{n+1}\to 1} \quad \text{for} \quad 1\le i \le s,
  \]
  and consider the corresponding ideal $I=\braket{f_1,\dots,f_s}$ in $S=\KK[x_1,\dots,x_n]$.  If $\widetilde{I}$ is of linear type, so is $I$. 
\end{prop}

\begin{proof}
  Consider the presentation 
  \[
  \psi: S[\bdT] \to \calR(I)
  \]
  defined by setting $\psi(T_i)=f_i t$ and denote by $J$ the kernel of $\psi$.  Similarly, we consider the presentation
  \[
  \widetilde{\psi}:\widetilde{S}[\bdT] \to \calR(\widetilde{I})
  \]
  defined by setting $\widetilde{\psi}(T_i)=f_it$ and denote by $\widetilde{J}$ the kernel of $\widetilde{\psi}$.

  Corresponding to $\widetilde{I}$, we can similarly define $\widetilde{\bdf}_\bdalpha$ and $\widetilde{\bdT}_{\bdalpha,\bdbeta}$.  
  Now, take any $\bdT_{\bdalpha,\bdbeta}\in J$ and consider the corresponding $\widetilde{\bdT}_{\bdalpha,\bdbeta}$ in $\widetilde{J}$.  Notice that $\widetilde{\bdf}_\bdalpha$ is the multiplication of $\bdf_\bdalpha$ with some power of $x_{n+1}$. Thus, $\bdf_\bdalpha=\left.\widetilde{\bdf}_\bdalpha\right|_{x_{n+1}\to 1}$.  Consequently, 
  \begin{align*}
    \left.\widetilde{\bdT}_{\bdalpha,\bdbeta}\right|_{x_{n+1}\to 1} = &  \left.
    \frac{\widetilde{\bdf}_\bdbeta}{\gcd(\widetilde{\bdf}_\bdalpha,\widetilde{\bdf}_\bdbeta)}\right|_{x_{n+1}\to
    1}\bdT_\bdalpha-\left.\frac{\widetilde{\bdf}_\bdalpha}{\gcd(\widetilde{\bdf}_\bdalpha,\widetilde{\bdf}_\bdbeta)}\right|_{x_{n+1}\to
    1}\bdT_\bdbeta \\
    = & \frac{\bdf_\bdbeta}{\gcd(\bdf_\bdalpha,\bdf_\bdbeta)}\bdT_\bdalpha-\frac{
    \bdf_\bdalpha}{\gcd(\bdf_\bdalpha,\bdf_\bdbeta)}\bdT_\bdbeta =  \bdT_{\bdalpha,\bdbeta}.
  \end{align*}

  On the other hand, since $\widetilde{\bdT}_{\bdalpha,\bdbeta}\in \widetilde{J}$, it can be generated in
  $\widetilde{S}[\bdT]$ by the
  linear parts of $\widetilde{J}$:
  \[
  \widetilde{\bdT}_{\bdalpha,\bdbeta}=\sum_{1\le i<j \le s} g_{i,j}\widetilde{\bdT}_{i,j}
  \quad\text{with} \quad g_{i,j}\in \widetilde{S}[\bdT].
  \]
  Now substitute $x_{n+1}$ with $1$ and we have
  \begin{align*}
    \bdT_{\bdalpha,\bdbeta}=&\left.\widetilde{\bdT}_{\bdalpha,\bdbeta}\right|_{x_{n+1}\to 1}
    =\sum_{1\le i<j \le s}\left. g_{i,j}\right|_{x_{n+1}\to 1}
    \left.\widetilde{\bdT}_{i,j}\right|_{x_{n+1}\to 1} \\
    = &\sum_{1\le i<j \le s}\left. g_{i,j}\right|_{x_{n+1}\to 1} {\bdT}_{i,j}
    \quad\text{with} \quad \left.g_{i,j}\right|_{x_{n+1}\to 1} \in S[\bdT].
  \end{align*}
  Thus, $\bdT_{\bdalpha,\bdbeta}$ can be generated in $S[\bdT]$ by the
  linear parts of $J$. This shows that $I$ is of linear type.
\end{proof}

\begin{rem}
The following partial converse of Proposition \ref{unwrap_cone} regarding inserting a free variable is incorrect. 
  \begin{center}
    \begin{minipage}{14cm}
      \it If $I$ is a squarefree monomial ideal in $S=\KK[x_1,\dots,x_n]$ of linear type with minimal monomial generators $G(I)=\Set{f_1,\dots,f_s}$, then the ideal $\widetilde{I}=\braket{f_1,\dots,f_{s-1},f_s\fbox{$x_{n+1}$}}$ in $S'=S[x_{n+1}]$ is also of linear type.
    \end{minipage}
  \end{center}
For a counterexample, one can compare the ideals $I$ and $\widetilde{I}$ in Example \ref{exam3}.
\end{rem}

\begin{prop}
  \label{even_cycles}
  Let $\Delta$ be a simplicial cycle of even length $s\ge 4$.  Then the facet ideal $I=I(\Delta)$ is not of linear type.
\end{prop}

\begin{proof}
  After Remark \ref{cone}, we may assume that the line graph $G$ of $I$ is a cycle of length $s$. 
  Now, apply Proposition \ref{unwrap_cone} inductively and use the map $\chi$ in Construction \ref{push-down} to substitute the variables in $D$ with $1$. We are reduced to the case where the squarefree monomial ideal is the edge ideal of a cycle graph $C_s$. This ideal is not of linear type by \cite[Corollary 3.2]{MR1335312}. Thus, Proposition \ref{unwrap_cone} implies that $I$ is not of linear type. 
\end{proof}

\section{Revisiting Villarreal's results}
In this section, we generalize some of Villarreal's results and ideas to squarefree monomial ideals of higher degrees. In particular, we will reprove Proposition \ref{Fouli-Lin} of Fouli and Lin.

\begin{prop}
  [Essentially Villarreal {\cite[Theorem 3.1]{MR1335312}}]
  \label{Vil3.1}
  Suppose the line graph $L(I)$ of the squarefree monomial ideal $I$ is a cycle graph. Take the variable set $D$ and the map $\chi$ in Construction \ref{push-down}. Then the defining ideal $J$ of $\calR(I)$ satisfies
  \[
 J= S[\bdT]J_1+S[\bdT]\cdot (\cup_{k=2}^{\infty} P_k),
  \]
  where
  \[
  P_k=\Set{\bdT_{\bdalpha,\bdbeta}| \chi(\bdf_\bdalpha)=\chi(\bdf_\bdbeta), \text{ for some $\bdalpha,\bdbeta\in\calI_k$}}.
  \]
\end{prop}

\begin{proof}
  When the length of the cycle is $3$, namely the line graph is a triangle, we may assume that the GCD of the three monomial generators is trivial. Now, it suffices to consider $\bdT_{\bdalpha,\bdbeta}$, where $\bdalpha=(i_1,\dots,i_k),\bdbeta=(j_1,\dots,j_k)\in\calI_k$ and $k\ge 2$. We only need to show that 
  \[
  \bdT_{\bdalpha,\bdbeta}\in S[\bdT]_1 J_{k-1}+S[\bdT]_{k-1}J_1+S P_{k}.
  \]
  Since the case when $\chi(\bdf_\bdalpha)=\chi(\bdf_\bdbeta)$ is trivial, we may assume that $\chi(\bdf_\bdalpha)\ne\chi(\bdf_\bdbeta)$. Thus, we may further assume that there is an $x\in D^c$ such that $x^a || \bdf_{\bdalpha}$ and $x^{b} || \bdf_\bdbeta$ for $b>a\ge 0$. Say, for instance, $x|f_{j_m}$. Then $\chi(f_{j_m})=xy$ for another $y\in D^c$. The variable $y$ belongs to exactly two monomial generators, one of which is $f_{j_m}$. If the other generator belongs to $\Set{f_{i_1},\dots,f_{i_k}}$, we may assume that it is $f_{i_l}$. Otherwise, we take $l=1$. Now $\chi(\lcm(\bdf_{\bdalpha},\bdf_\bdbeta))$ is a multiple of $\chi(f_{j_m}{\bdf_{\bdalpha\setminus l}})$ where $\bdf_{\bdalpha\setminus l}:=f_{i_1}\cdots \widehat{f_{i_l}}\cdots f_{i_k}=\bdf_\bdalpha/f_{i_l}$. We will similarly define $\bdT_{\bdalpha\setminus l}$.

  It follows that for each $z\in D^c$, 
  \begin{equation}
    \deg_z \lcm(\bdf_\bdalpha,\bdf_\bdbeta) \ge \deg_z (f_{j_m} \bdf_{\bdalpha\setminus l}).
    \label{less} \tag{\ddag}
  \end{equation}
  Notice that, if $z\in D^c$ and $z'\in D$ is one of its shadows, then $\deg_{z'}\bdf_\bdalpha=\deg_{z}\bdf_\bdalpha$. Thus the condition \eqref{less} also holds for shadows. Clearly, condition \eqref{less} holds for free variables as well. 
  
  Thus, $\lcm(\bdf_\bdalpha,\bdf_\bdbeta)$ is a multiple of $f_{j_m} \bdf_{\bdalpha\setminus l}$ and we may write $\lcm(\bdf_\bdalpha,\bdf_\bdbeta)=A f_{j_m} \bdf_{\bdalpha\setminus l}$ for some $A\in S$.
Now 
\begin{align*}
  \bdT_{\bdalpha,\bdbeta} & = \dfrac{\lcm(\bdf_{\bdalpha},\bdf_\bdbeta)}{\bdf_\bdalpha}\bdT_\bdalpha - \dfrac{\lcm(\bdf_{\bdalpha},\bdf_\bdbeta)}{\bdf_\bdbeta}\bdT_\bdbeta \\
  & = \left( \dfrac{\lcm(\bdf_\bdalpha,\bdf_\bdbeta)}{\bdf_\bdalpha}T_{i_l}-AT_{j_m}\right) \bdT_{\bdalpha\setminus l} + 
  \left(A\bdT_{\bdalpha\setminus l} - \dfrac{\lcm(\bdf_\bdalpha,\bdf_\bdbeta)}{\bdf_\bdbeta}\bdT_{\bdbeta\setminus m}\right)T_{j_m} \\
  & = \mu\left(\dfrac{\lcm(f_{i_l},f_{j_m})}{f_{i_l}}T_{i_l}-\dfrac{\lcm(f_{i_l},f_{j_m})}{f_{j_m}}T_{j_m}\right)\bdT_{\bdalpha\setminus l} \\
  & \qquad + \lambda \left( \dfrac{ \lcm(\bdf_{\bdalpha\setminus l},\bdf_{\bdbeta\setminus m})}{\bdf_{\bdalpha\setminus l}}\bdT_{\bdalpha\setminus l} - \dfrac{\lcm(\bdf_{\bdalpha\setminus l},\bdf_{\bdbeta\setminus m})}{\bdf_{\bdbeta\setminus m}}\bdT_{\bdbeta\setminus m}\right)T_{j_m} \\
  & = \mu \bdT_{i_l,j_m}\bdT_{\bdalpha\setminus l}+\lambda \bdT_{\bdalpha\setminus l, \bdbeta\setminus m} T_{j_m} \quad \in \quad S[\bdT]_{k-1}J_1+S[\bdT]_{1} J_{k-1} 
\end{align*}
for suitable $\mu,\lambda\in S$. To justify the existence of $\mu$, we observe that
\[
\frac{\lcm(\bdf_\bdalpha,\bdf_\bdbeta)/\bdf_\bdalpha}{\lcm(f_{i_l},f_{j_m})/f_{i_l}} = \frac{A}{\lcm(f_{i_l},f_{j_m})/f_{j_m}}.
\]
Thus, $ (\lcm(\bdf_\bdalpha,\bdf_\bdbeta)/\bdf_\bdalpha)T_{i_l}-AT_{j_m}\in J$. However, $(\lcm(f_{i_l},f_{j_m})/f_{i_l})T_{i_l}-(\lcm(f_{i_l},f_{j_m})/f_{j_m})T_{j_m}$ corresponds to the minimal relation between $f_{i_l}t$ and $f_{j_m}t$. This guarantees the existence of $\mu$. The situation for $\lambda$ is similar.
\end{proof}

\begin{exam}
  Villarreal \cite[Example 3.1]{MR1335312} considered the ideal $I$ generated by $f_1=x_1x_2x_3$, $f_2=x_2x_4x_5$, $f_3=x_5x_6x_7$ and $f_4=x_3x_6x_7$ in $S=\KK[x_1,\dots,x_7]$. The defining ideal is generated by
  \[
  x_3T_3-x_5T_4,x_6x_7T_1-x_1x_2T_4,x_6x_7T_2-x_2x_4T_3,x_4x_5T_1-x_1x_3T_2,x_4T_1T_3-x_1T_2T_4.
  \]
  The line graph $L(I)$ of $I$ is a square. We can take $D=\Set{x_1,x_4,x_6}$ and $D^c=\Set{x_2,x_3,x_5,x_7}$ and define the corresponding map $\chi$. The variable $x_6$ is a shadow of $x_7\in D^c$. As for the generator $x_4T_1T_3-x_1T_2T_4\in P_2$, we notice that $\chi(f_1f_3)=\chi(f_2f_4)=x_2x_3x_5x_7$.
\end{exam}

\begin{rem}
  \label{tree}
  If $f\in G(I)$ corresponds to a leaf in the line graph $L(I)$ of $I$, then $f$ is an $M$-element of $I$. Thus, when $L(I)$ is a forest, the ideal $I$ is of linear type by Corollary \ref{res0}.
\end{rem}

\begin{proof}
  [Proof of Proposition \ref{Fouli-Lin}]
  By Remarks \ref{connected-component} and \ref{tree}, we may assume that the line graph $L(I)$ is a cycle. When this cycle is a triangle, we may additionally assume that the GCD of the three monomial generators of $I$ is trivial.
  Now, take a $\bdT_{\bdalpha,\bdbeta}\in P_k$ for $k\ge 2$ in Proposition \ref{Vil3.1}. Say $\bdalpha=(i_1,\dots,i_k)$ and $\bdbeta(j_1,\dots,j_k)$. We may further assume that $\Set{i_1,\dots,i_k}\cap \Set{j_1,\dots,j_k}=\emptyset$.

  We first consider the case that $I'=\braket{f_{i_1},\dots,f_{i_k},f_{j_1},\dots,f_{j_k}}=I$ and the line graph $L(I')$ is the whole cycle. We may write that $\Set{f_{i_1},\dots,f_{i_k}}=\Set{g_1,\dots,g_l}$, where the $g_p\ne g_q$ for $p\ne q$.
  Suppose for contradiction that $\gcd(g_p,g_q)\ne 1$ for some $p\ne q$. Then $x| \gcd(g_p,g_q)$ for some $x\in D^c$. Since $x$ divides up to two monomial generators of $I$, $x$ divides none of the $f_{j_r}$, $r=1,\dots,k$. This contradicts the assumption that $\chi(\bdf_\bdalpha)=\chi(\bdf_{\bdbeta})$.  
  Hence $\gcd(g_p,g_q)=1$ for $p\ne q$. 
  
  Since the line graph $L(I)$ is the whole cycle, we must have $\chi(g_1\cdots g_l)=\prod_{x\in D^c} x$. But $\deg(\chi(g_1\cdots g_l))=2l$ is an even number, while $\# D^c$ is the length of the cycle which is an odd number. This is impossible.

  Thus $I'\ne I$ and the line graph $L(I')$ of $I'$ must be a forest. Now $I'$ is of linear type and $\bdT_{\bdalpha,\bdbeta}$ is a linear combination of the linear forms. This shows that $I$ is of linear type.
\end{proof}

The following result is a partial generalization of Villarreal's \cite[Proposition 3.1]{MR1335312}.

\begin{prop}
  \label{even_generators}
  Let $I$ be a squarefree monomial ideal, such that the line graph $L(I)$ of $I$ is an even cycle of length $s=2k\ge 4$. Suppose that the generators $f_1,\dots,f_{2k}$ are ordered such that they induce a monomial walk in $L(I)$. Then the defining ideal $J$ is generated by $J_1$ and $\bdT_w:=\bdT_{\bdw_1,\bdw_2}$ where $\bdw_1=(1,3,\dots,2k-1)$ and $\bdw_2=(2,4,\dots,2k)\in \calI_k$. 
\end{prop}

\begin{proof}
  As in our (re)proof for Proposition \ref{Fouli-Lin}, we may consider the relation $\bdT_{\bdalpha,\bdbeta}\in P_r$ for $\bdalpha=(i_1,\dots,i_r)$ and $\bdbeta=(j_1,\dots,j_r)\in \calI_s$ with $\Set{i_1,\dots,i_r}\cap \Set{j_1,\dots,j_r}=\emptyset$. We may as well assume that $\Set{f_{i_1},\dots,f_{i_r}}=\Set{g_1,\dots,g_l}$ with $\gcd(g_p,g_q)=1$ for $p\ne q$ and $\chi(g_1\cdots g_l)=\prod_{x\in D^c}x$. Thus, $l=k$ and $\Set{g_1,\dots,g_k}$ equals either $\Set{f_1,f_3,\dots,f_{2k-1}}$ or $\Set{f_2,f_4,\dots,f_{2k}}$. Without loss of generality, we assume that it is $\Set{f_1,f_3,\dots,f_{2k-1}}$.

  Notice that $\gcd(f_1,f_2)\ne 1$. Take $x:=\chi(\gcd(f_1,f_2))\in D^c$. Since $x$ divides only $f_1$ and $f_2$, 
  \[
  \#\Set{p|i_p=1\text{ for } 1\le p \le r}=\deg_x \chi(\bdT_\bdalpha)
    =\deg_x\chi(\bdT_\bdbeta)=\#\Set{q|j_q=2\text{ for }1\le q \le r}.
    \]
Denote the above multiplicity by $\gamma$. Then one can easily show that $\#\Set{p|i_p=o \text{ for } 1\le p \le r}=\gamma$ for all odd indices $o$ in $[2k]$. Thus, $r=k\gamma$ and $\bdf_\bdalpha=(f_1f_3\cdots f_{2k-1})^\gamma$. Similarly, we know $\bdf_\bdbeta=(f_2f_4\cdots f_{2k})^\gamma$. It is clear that $\bdT_{\bdalpha,\bdbeta}$ is divisible by $\bdT_w$.
\end{proof}

\begin{df}
  \label{patch}
  Let $\Delta$ be a simplicial complex and $F_1$ and $F_2$ be two adjacent facets in $\Delta$. A \Index{patch that covers $F_1$ and $F_2$} is a new facet $G$ to $\Delta$ ($G$ is not comparable with any of the existing facets of $\Delta$), such that
  \[
  G \subset \bigg(F_1\cup F_2\bigg) \setminus \bigcup_{F\in \calF(\Delta)\setminus \Set{F_1,F_2}} F.
  \]
\end{df}

In Example \ref{exam3}, face $G$ of $\Delta$ corresponds to a patch that covers $F_2$ and $F_3$ in $\Delta'$. But $\widetilde{G}$ of $\widetilde{\Delta}$ does not.

\begin{thm}
  \label{even1patch}
  Suppose the simplicial complex $\Delta'$ is a linear cycle of even length $2k\ge 4$. Let $G$ be a patch that covers $F_1$ and $F_2$ of $\calF(\Delta')$. Then the simplicial complex $\Delta:=\braket{\Delta',G}$ is of linear type.
\end{thm}

\begin{proof}
  We may assume that the sequence $F_1,F_2,\dots,F_{2k},F_1$ corresponds to a closed walk in $L(\Delta')$. We will also write $F_{2k+1}=G$. Suppose each $F_i$ corresponds to the squarefree monomial $f_i$ in $S$ and $G$ corresponds to the squarefree monomial $f_{2k+1}=g$. Now $I'=\braket{f_1,\dots,f_{2k}}$ and $I=\braket{I,g}$ are the facet ideals of $\Delta'$ and $\Delta$ respectively.

  To show $I$ is of linear type, it suffices to check a relation $\bdT_{\bdalpha,\bdbeta}$ with $\bdalpha=(i_1,\dots,i_l)$ and $\bdbeta=(j_1,\dots,j_l)\in \calI_l$ for $l\ge 2$. We may assume that $\calI_\bdalpha:=\Set{i_1,\dots,i_l}\cap \calI_\bdbeta:=\Set{j_1,\dots,j_l}=\emptyset$ and deal with the following cases.
  \begin{enumerate}[a]
    \item \label{case_a} $ 2k+1\notin (\calI_\bdalpha \cup \calI_\bdbeta)$. In this case, $\bdT_{\bdalpha,\bdbeta}$ deals with relations for generators of $I'$. By Proposition \ref{even_generators}, we may assume that $l=k$ and \[
      \bdT_{\bdalpha,\bdbeta}=\bdT_w=\dfrac{\lcm(\bdf_{\bdw_1},\bdf_{\bdw_2})}{\bdf_{\bdw_1}}\bdT_{\bdw_1}-\dfrac{\lcm(\bdf_{\bdw_1},\bdf_{\bdw_2})}{\bdf_{\bdw_2}}\bdT_{\bdw_2}.
      \]
      Notice that in 
      \[
      \bdT_{2,2k+1}=\dfrac{\lcm(f_2,f_{2k+1})}{f_2}T_2-\dfrac{\lcm(f_2,f_{2k+1})}{f_{2k+1}}T_{2k+1},
      \]
      $\lcm(f_2,f_{2k+1})/{f_2}$ is a product of some free variables of $f_1$. Thus 
      \[
      \dfrac{\lcm(f_2,f_{2k+1})}{f_2} T_2 \, \left| \, \dfrac{\lcm(\bdf_{\bdw_1},\bdf_{\bdw_2})}{\bdf_{\bdw_2}}\bdT_{\bdw_2}\right. ,
      \]
      and we are reduced to consider the next case.
    \item $ 2k+1\in (\calI_\bdalpha \cup \calI_\bdbeta)$. By symmetry, we have the following several subcases.
      \begin{enumerate}[i]
        \item $\Set{1,2}\not\subset \calI_\bdalpha \cup \calI_\bdbeta$.  In this case, the underlying line graph of $\bdT_{\bdalpha,\bdbeta}$ is a forest. Hence $\bdT_{\bdalpha,\bdbeta}$ is a linear combination of linear forms.
        \item $\Set{1,2,2k+1}\cap \calI_\bdalpha=\Set{1,2k+1}$ and $\Set{1,2,2k+1}\cap \calI_\bdbeta=\Set{2}$. As in case \ref{case_a}, we can take use of $\bdT_{2,2k+1}$. Since 
          \[
          \lcm(\bdf_\bdalpha,\bdf_\bdbeta)/\bdf_\bdbeta=A\lcm(f_2,f_{2k+1})/f_2
          \]
          for some $A\in S$, we have
          \[
          \bdT_{\bdalpha,\bdbeta}= - A\bdT_{\bdbeta\setminus 2} \bdT_{2,2k+1}+
          \left( \dfrac{\lcm(\bdf_\bdalpha,\bdf_\bdbeta)}{\bdf_\bdalpha}\bdT_{\bdalpha\setminus 2k+1}- \dfrac{\lcm(f_2,f_{2k+1})}{f_{2k+1}} A \bdT_{\bdbeta\setminus 2}  \right)  T_{2k+1}.
          \]
          The first summand is divisible by the linear form $\bdT_{2,2k+1}$ while the second summand is divisible by $T_{2k+1}$. Thus we can decrease the degree $l$ and prove by induction.
        \item $\Set{1,2,2k+1}\cap \calI_\bdalpha=\Set{1,2}$ and $\Set{1,2,2k+1}\cap \calI_\bdbeta=\Set{2k+1}$. Notice that 
          \[
          {\lcm(\bdf_{\bdalpha},\bdf_\bdbeta)}/{\bdf_\bdbeta}=A {\lcm(f_2,f_{2k+1})}/{f_2} 
          \]
          for some $A\in S$. Thus  
          \[
          \bdT_{\bdalpha,\bdbeta}= - A\bdT_{\bdbeta\setminus 2k+1} \bdT_{2,2k+1}+
          \left( \dfrac{\lcm(\bdf_\bdalpha,\bdf_\bdbeta)}{\bdf_\bdalpha}\bdT_{\bdalpha\setminus 2}- \dfrac{\lcm(f_2,f_{2k+1})}{f_{2}} A \bdT_{\bdbeta\setminus 2k+1}  \right)  T_{2}.
          \]
          The first summand is divisible by the linear form $\bdT_{2,2k+1}$ while the second summand is divisible by $T_{2}$. Thus we can decrease the degree $l$ and prove by induction.
        \item $\Set{1,2,2k+1}\subset \calI_\bdalpha$. We can argue as in the proof for Proposition \ref{Vil3.1} by noticing that $\lcm(\bdf_\bdalpha,\bdf_\bdbeta)=A f_1 \bdf_{\bdbeta\setminus p}$ for some $p\in \calI_\bdbeta$. \qedhere
      \end{enumerate}
  \end{enumerate}
\end{proof}

\begin{df}
  \label{compatible-patches}
  Let $\calP=\Set{p_{i}^{j_i,k_i}}_{i=1}^q$ be a finite set of patches of a simplicial complex $\Delta$ such that $p_i^{j_i,k_i}$ covers $F_{j_i}$ and $F_{k_i}$. $\calP$ is called a \Index{system of compatible patches} if $\Set{j_{i_1},k_{i_1}}\ne \Set{j_{i_2},k_{i_2}}$  and $p_{i_1}^{j_{i_1},k_{i_1}}\cap p_{i_2}^{j_{i_2},k_{i_2}}=\emptyset$ for $i_1\ne i_2$.
\end{df}

\begin{conj}
  Let $\Delta$ be a simplicial complex whose line graph is a cycle of length $l$ and $\calP=\Set{p_i^{j_i,k_i}}_{i=1}^q$ is a system of compatible patches. If $l+q$ is odd, then $\braket{\Delta,\calP}$ is of linear type.
\end{conj}

\section{Villarreal class}

According to Zheng \cite{1196.13003}, a facet $F$ of $\Delta$ is called a \Index{good leaf} if this $F$ is a leaf of each subcomplex $\Gamma$ of $\Delta$ to which $F$ belongs. An order $F_1,\dots,F_s$ of facets is called a \Index{good leaf order} if $F_i$ is a good leaf of $\braket{F_1,\dots,F_i}$ for each $i=1,\dots,s$. It is known that a simplicial complex is a forest if and only if it has a good leaf order; see \cite[Corollary 1.11]{MR2968912} or \cite{arxiv-1307.2190}.

\begin{lem}
  \label{equiv}
  The following two conditions are equivalent:
  \begin{enumerate}[a]
    \item The facet $F_1$ is a good leaf of $\Delta$.
    \item The squarefree monomial $\bdx_{F_1}$ is an $M$-element of $I(\Delta)$.
  \end{enumerate}
\end{lem}

\begin{proof}
  This result follows from the proofs of \cite[Proposition 3.11]{1196.13003} and \cite[Proposition 1.12]{MR2968912}. 
\end{proof}

We will take simplexes as simplicial cycles of length $1$. 

\begin{df}
  \label{Villarreal-type}
  Let $\calV$ be the class of simplicial complexes minimal with respect to the following properties:
  \begin{itemize}
    \item Disjoint simplicial  cycles of odd lengths are in $\calV$.
    \item $\calV$ is closed under the operation of attaching good leaves.
  \end{itemize}
  We shall call $\calV$ the \Index{Villarreal class}. When a simplicial complex $\Delta$ is in $\calV$, we say $\Delta$ and its facet ideal $I(\Delta)$ are \Index{of Villarreal type}. 
\end{df}

According to Villarreal \cite[Corollary 3.2]{MR1335312}, a connected simple graph $G$ belongs to $\calV$ if and only if its edge ideal $I(G)$ is of linear type. 

\begin{thm}
  Squarefree monomial ideals of Villarreal type are of linear type.
\end{thm}

\begin{proof}
  It follows from Corollaries \ref{res1}, \ref{odd_cycles}, Remark \ref{connected-component} as well as Lemma \ref{equiv}.
\end{proof}

\begin{rem}
  The class of monomial ideals of Villarreal type $\calV$ is a proper subclass of the class of squarefree monomial ideals of linear type $\calL$. See, for instance, the Example \ref{exam3} below.
\end{rem}

\begin{rem}
  Simplicial forests are obviously of Villarreal type. Thus, their facet ideals are of linear type. This fact is already known, e.g., after combining \cite[Theorem 1.14]{MR2968912} with \cite[Theorem 2.4.i]{MR1666661}. On the other hand, the facet ideals of quasi-forests are not necessarily of linear type. Recall that a connected simplicial complex $\Delta$ is called a \Index{quasi-tree}, if there exists an order  $F_1,\dots, F_s$ of the facets, such that $F_i$ is a leaf of $\braket{F_1,\dots,F_i}$ for each $i=1,\dots,s$. Such an order is called a \Index{leaf order}. A simplicial complex $\Delta$ with the property that every connected component is a quasi-tree is called a \Index{quasi-forest}. The following simplicial complex $\Gamma$, which was originally given in \cite{MR2968912}, is an example of quasi-tree.
  \[
  \Gamma=\begin{tabular}{c}
    \resizebox{!}{1in}
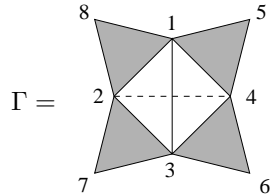
{\begin{picture}(0,0)%
\includegraphics{D6.pstex}%
\end{picture}%
\setlength{\unitlength}{3947sp}%
\begingroup\makeatletter\ifx\SetFigFont\undefined%
\gdef\SetFigFont#1#2#3#4#5{%
  \reset@font\fontsize{#1}{#2pt}%
  \fontfamily{#3}\fontseries{#4}\fontshape{#5}%
  \selectfont}%
\fi\endgroup%
\begin{picture}(3007,2990)(4530,-4586)
\end{picture}%
}
  \end{tabular}
  \]
  Using Macaulay2 \cite{M2}, we know its facets ideal
  \[
  I(\Gamma)=\braket{x_1x_2x_3x_4,x_1x_4x_5,x_1x_2x_8,x_2x_3x_7,x_3x_4x_6}
  \]
  is not of linear type.
\end{rem}

\section{Other cycles and examples}
Recall that an alternating sequence of distinct vertices and facets 
\begin{equation}
  v_1,F_1,\dots,v_s,F_s,v_{s+1}=v_1 \label{special_cycle} \tag{\dag}
\end{equation}
in a simplicial complex $\Delta$ is called a \Index{(hyper)cycle} or a \Index{Berge cycle} if $v_i,v_{i+1}\in F_i$ for all $i$. When $\Delta$ is one-dimensional, we can treat it as a graph. Then a Berge cycle is exactly a cycle or walk in the classic sense. 

According to \cite{MR2434285}, a Berge cycle is called \Index{special} if no facets contains more than two connecting vertices of this sequence. Special cycles of length $2$ correspond to facets whose intersection is at least $1$-dimensional. And by \cite[Theorem 3.2]{MR2434285}, a simplicial complex is a forest if and only if it contains no special cycle of length $\ge 3$.

Notice that classic cycle graphs are always special. Thus, with the help of Lemma \ref{simpl_cycle}, Remark \ref{cone} and Construction \ref{push-down}, we can always carefully pick up the connecting vertices and construct a special cycle from a given simplicial cycle.  On the other hand, the converse is false, i.e., the underlying simplicial complex of a special cycle is not necessarily a simplicial cycle, as illustrated by the ideal $I'$ in Example \ref{exam2}. Nevertheless, we have the following result.

\begin{prop}
  \label{equiv4}
  Let $\Delta$ be a connected simplicial complex which is not a cone. Suppose
  $s=\#\calF(\Delta)\ge 4$. Then the following statements are equivalent. 
  \begin{enumerate}[a]
    \item $\Delta$ is a simplicial cycle.
    \item $\Delta$ is a linear cycle.
    \item $\Delta$ induces a special cycle of length $s$, but cannot induces any smaller special cycles of length $k$ with $s-1\ge k \ge 3$.
    \item $\Delta$ induces a Berge cycle of length $s$, but cannot induces any smaller Berge cycles of length $k$ with $s-1\ge k \ge 3$.
  \end{enumerate}
\end{prop}

\begin{proof}
  It follows easily from Lemma \ref{simpl_cycle} and the above comments that $(a)\Leftrightarrow (b) \Rightarrow (c)$ and $(b) \Rightarrow (d)$. Now, we first show $(c)\Rightarrow (b)$. 

  Let $V(\Delta)$ be the vertex set of $\Delta$ and assume that \eqref{special_cycle} gives a special cycle in $\Delta$. Observe that this special cycle naturally induces a cycle $C_s$ in the line graph $L(\Delta)$ of $\Delta$.

  Suppose for contradiction that $\Delta$ is not a linear cycle. Then there is some vertex 
  \[
  v\in
  V(\Delta)\setminus \Set{v_1=v_{s+1},v_2,\dots,v_s}
  \]
  that induces a complete subgraph $K_r$ in $L(\Delta)$, which is not contained in the previous $C_s$. Let $r$ be maximal with respect to this vertex $v$. Since $\Delta$ is not a cone, $L(\Delta)$ properly contains this $K_r$. Thus, $L(\Delta)$ contains an arc of length $k-1$ $(\ge 2)$ from this $C_s$, which intersects $K_r$ at exactly two vertices. Without loss of generality, we may assume that this arc corresponds to the subsequence of facets $F_1,F_2,\dots,F_{k}$ from the special cycle \eqref{special_cycle} and $v\in F_1\cap F_{k}$. Thus, we have a special cycle 
  \begin{equation}
    v,F_1,v_2,F_2,\dots,v_{k},F_{k},v 
    \label{short-cycle} \tag{$\star$}
\end{equation}
  of shorter length, which leads to a contradiction.

  The proof for $(d)\Rightarrow (b)$ is almost identical. We only need to choose $v\in V(\Delta)$ without requiring that $v\not\in \Set{v_1,\dots,v_s}$. Nevertheless, $v$ is necessarily distinct from $v_2,\dots,v_k$ in \eqref{short-cycle}. Thus, it is a Berge cycle of shorter length.
\end{proof}

\begin{exam}
  \label{exam1}
  Let $S=\KK[x_1,\dots,x_{10}]$, $I=\braket{x_1x_2,x_2x_3,x_3x_4x_8,x_4x_5x_6x_7,x_1x_5x_{9}}$ and $f=x_9x_{10}$.  For the underlying simplicial complex $\Delta$, the facets form a special cycle of length 5:
  \[
  1,\Set{1,2},2,\Set{2,3},3,\Set{3,4,8},4,\Set{4,5,6,7},5,\Set{1,5,9},1.
  \] 
  Actually, $\Delta$ is a simplicial cycle, although it is not a forest. Thus, we still have $\Delta\in \calV$. In particular, $I=I(\Delta)$ is still of linear type. Notice that $\Set{9,10}$ is a good leaf that can be attached to $\Delta$. The corresponding new ideal $I+(f)$ is also of linear type.
\end{exam} 

The next example is a modification of the previous one.

\begin{exam}
  \label{exam2}
  Let $S=\KK[x_1,\dots,x_{10}]$, $I'=\braket{x_1x_2,x_2x_3\fbox{$x_7$},x_3x_4x_8,x_4x_5x_6x_7,x_1x_5x_{9}}$ and $f=x_9x_{10}$.  For the underlying simplicial complex $\Delta'$, the facets form a special cycle of length 5:
  \[
  1,\Set{1,2},2,\Set{2,3,\fbox{7}},3,\Set{3,4,8},4,\Set{4,5,6,7},5,\Set{1,5,9},1.
  \] 
  This is not a simplicial  cycle, since $\Delta'$ also contains a special cycle of length $3$:
  \[
  7,\Set{2,3,7},3,\Set{3,4,8},4,\Set{4,5,6,7},7,
  \]
  as well as a special cycle of length $4$:
  \[
  1,\Set{1,2},2,\Set{2,3,7},7,\Set{4,5,6,7},5,\Set{1,5,9},1.
  \] 
  Although we can still find $\Set{9,10}$ as a good leaf to attach to $\Delta'$, neither $I'$ nor $I'+\braket{f}$ is of linear type.
\end{exam}

\begin{exam}
  \label{exam3}
  This example comes from \cite[3.12]{MR2284286}.
  \[
  \Delta=\begin{tabular}{c}
    \resizebox{!}{1in}{\input{D0.pstex_t}}
  \end{tabular}
  \qquad
  \qquad
  \Delta'=\begin{tabular}{c}
    \resizebox{!}{1in}{\input{D1.pstex_t}}
  \end{tabular}
  \qquad
  \qquad
  \widetilde{\Delta}=\begin{tabular}{c}
    \resizebox{!}{1in}{\input{D2.pstex_t}}
  \end{tabular}
  \]
  $\Delta$ is not a simplicial cycle and has no leaves.
  Corresponding to $\Delta$, we consider the ideal 
  \[
  I=\braket{x_1x_2x_5x_6,x_2x_3x_7x_8,x_3x_4x_9x_{10},x_1x_4x_{11}x_{12},x_3x_8x_9}
  \subset \KK[x_1,\dots,x_{12}].
  \]
  The simplicial complex $\Delta$ is not of Villarreal type. Using Macaulay2 \cite{M2}, we know $I$ is still of linear type, as expected in Theorem \ref{even1patch}.

  Notice that the subcomplex $\Delta'$ is a simplicial cycle of length $4$. The existence of this even-length cycle in $\Delta$ does not prevent $I$ from being of linear type.

  On the other hand, the ideal
  \[
  \widetilde{I}=\braket{x_1x_2x_5x_6,x_2x_3x_7x_8,x_3x_4x_9x_{10},x_1x_4x_{11}x_{12},x_3x_8x_9\fbox{$x_{13}$}}
  \subset \KK[x_1,\dots,x_{12},x_{13}],
  \]
  which is modified from $I$ by inserting a free variable, is not of linear type.
\end{exam}

\begin{exam}
  \label{exam4}
  To illustrate the proof of Proposition \ref{equiv4}, we consider the simplicial complex $\Delta$ and $\Delta'$ in Example \ref{exam3}.
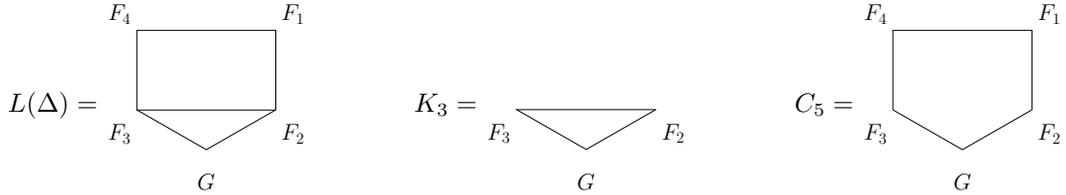
  \[
  L(\Delta)=\begin{tabular}{c}
    \resizebox{!}{1in}{\input{D3.pstex_t}}
  \end{tabular}
  \qquad
  \qquad
  K_3=\begin{tabular}{c}
    \resizebox{!}{1in}{\input{D4.pstex_t}}
  \end{tabular}
  \qquad
  \qquad
  C_5=\begin{tabular}{c}
    \resizebox{!}{1in}{\input{D5.pstex_t}}
  \end{tabular}
  \]
  It is clear that 
  \[
  1,F_1,2,F_2,8,G,9,F_3,4,F_4,1
  \]
  gives a special cycle of length $5$. Its line graph is $L(\Delta)$ and the previous special cycle corresponds to the cycle $C_5$ in $L(\Delta)$. Now the vertex $3$ contributes to the complete subgraph $K_3$ in $L(\Delta)$. We will take the arc $F_2F_1F_4F_3$ which intersects $K_3$ at the vertices $F_2$ and $F_3$. From this, we get the special cycle of length $4$ in $\Delta$:
  \[
  3, F_2,2,F_1,1,F_4,4,F_3,3.
  \]
  It corresponds to the subcomplex $\Delta'$.
\end{exam}


\end{document}